\documentclass[12pt]{amsart}
\usepackage{amsmath,amsthm,amsfonts,amssymb,mathrsfs}
\date{\today}

\usepackage{color}

\newtheorem{theorem}{Theorem}
\newtheorem{proposition}[theorem]{Proposition}
\newtheorem{corollary}[theorem]{Corollary}
\newtheorem{lemma}[theorem]{Lemma}

\theoremstyle{definition}
\newtheorem{question}[theorem]{Question}
\newtheorem{example}[theorem]{Example}
\newtheorem{definition}[theorem]{Definition}

\usepackage{hyperref}

\begin{document}

\title[Brandt extensions and primitive topological inverse
semigroups] {Brandt extensions and primitive topological inverse
semigroups}

\author{Tetyana~Berezovski}
\address{Department of Math/CSC,
Saint Joseph's University 5600, City Avenue, Philadelphia PA
19131, U.S.A.} \email{tberezov@sju.edu}

\author{Oleg~Gutik}
\address{Department of Mathematics, Ivan Franko Lviv National
University, Universytetska 1, Lviv, 79000, Ukraine}
\email{o\underline{\hskip5pt}\,gutik@franko.lviv.ua,
ovgutik@yahoo.com}

\author{Kateryna~Pavlyk}
\address{Pidstrygach Institute for Applied
Problems of Mechanics and Mathematics of National Academy of
Sciences, Naukova 3b, Lviv, 79060, Ukraine and Department of
Mathematics, Ivan Franko Lviv National University, Universytetska
1, Lviv, 79000, Ukraine} \email{kpavlyk@yahoo.co.uk}



\begin{abstract}
In the paper we study (countably) compact and (absolutely)
$H$-closed primitive topological inverse semigroups. We describe
the structure of compact and countably compact primitive
topological inverse semigroups and show that any countably compact
primitive topological inverse semigroup embeds into a compact
primitive topological inverse semigroup.
\end{abstract}

\maketitle


In this paper all spaces are Hausdorff.

A semigroup is a non-empty set with a binary associative
operation. A semigroup $S$ is called \emph{inverse} if for any
$x\in S$ there exists a unique $y\in S$ such that $x\cdot y\cdot
x=x$ and $y\cdot x\cdot y=y$. Such an element $y$ in $S$ is called
\emph{inverse} to $x$ and denoted by $x^{-1}$. The map defined on
an inverse semigroup $S$ which maps to any element $x$ of $S$ its
inverse $x^{-1}$ is called the \emph{inversion}.

A {\it topological semigroup} is a Hausdorff topological space
with a jointly continuous semigroup operation. A topological
semigroup which is an inverse semigroup is called an \emph{inverse
topological semigroup}. A \emph{topological inverse semigroup} is
an inverse topological semigroup with continuous inversion. A {\it
topological group} is a topological space with a continuous group
operation and an inversion. We observe that the inversion on a
topological inverse semigroup is a homeomorphism (see
\cite[Proposition~II.1]{EberhartSelden1969}). A Hausdorff topology
$\tau$ on a (inverse) semigroup $S$ is called (\emph{inverse})
\emph{semigroup} if $(S,\tau)$ is a topological (inverse)
semigroup.

Further we shall follow the terminology of \cite{CHK, CP,
Engelking1989, Howie1995, Petrich1984}. If $S$ is a semigroup,
then by $E(S)$ we denote the band (the subset of idempotents) of
$S$, and by $S^1$ [$S^0$] we denote the semigroup $S$ with the
adjoined unit [zero] (see \cite[p.~2]{Howie1995}). Also if a
semigroup $S$ has zero $0_S$, then for any $A\subseteq S$ we
denote $A^*=A\setminus\{ 0_S\}$. If $Y$ is a subspace of a
topological space $X$ and $A\subseteq Y$, then by
$\operatorname{cl}_Y(A)$ we denote the topological closure of $A$
in $Y$. The set of positive integers is denoted by $\mathbb{N}$.

If $E$ is a semilattice, then the semilattice operation on $E$
determines the partial order $\leqslant$ on $E$: $$e\leqslant
f\quad\text{if and only if}\quad ef=fe=e.$$ This order is called
{\em natural}. An element $e$ of a partially ordered set $X$ is
called {\em minimal} if $f\leqslant e$  implies $f=e$ for $f\in
X$. An idempotent $e$ of a semigroup $S$ without zero (with zero)
is called \emph{primitive} if $e$ is a minimal element in $E(S)$
(in $(E(S))^*$).

Let $S$ be a semigroup with zero and $I_\lambda$ be a set of
cardinality $\lambda\geqslant 1$. On the set
$B_{\lambda}(S)=\left(I_\lambda\times S\times
I_\lambda\right)\cup\{ 0\}$ we define the semigroup operation as
follows
 $$
 (\alpha,a,\beta)\cdot(\gamma, b, \delta)=
  \begin{cases}
    (\alpha, ab, \delta), & \text{ if } \beta=\gamma; \\
    0, & \text{ if } \beta\ne \gamma,
  \end{cases}
 $$
and $(\alpha, a, \beta)\cdot 0=0\cdot(\alpha, a, \beta)=0\cdot
0=0,$ for all $\alpha, \beta, \gamma, \delta\in I_\lambda$ and $a,
b\in S$. If $S=S^1$ then the semigroup $B_\lambda(S)$ is called
the {\it Brandt $\lambda-$extension of the semigroup}
$S$~\cite{GutikPavlyk2001}. Obviously, ${\mathcal J}=\{
0\}\cup\{(\alpha, {\mathscr O}, \beta)\mid {\mathscr O}$ is the
zero of $S\}$ is an ideal of $B_\lambda(S)$. We put
$B^0_\lambda(S)=B_\lambda(S)/{\mathcal J}$ and we shall call
$B^0_\lambda(S)$ the {\it Brandt $\lambda^0$-extension of the
semigroup $S$ with zero}~\cite{GutikPavlyk2006}. Further, if
$A\subseteq S$ then we shall denote $A_{\alpha,\beta}=\{(\alpha,
s, \beta)\mid s\in A \}$ if $A$ does not contain zero, and
$A_{\alpha,\beta}=\{(\alpha, s, \beta)\mid s\in A\setminus\{ 0\}
\}\cup \{ 0\}$ if $0\in A$, for $\alpha, \beta\in I_{\lambda}$. If
$\mathcal{I}$ is a trivial semigroup (i.e. $\mathcal{I}$ contains
only one element), then by ${\mathcal{I}}^0$ we denote the
semigroup $\mathcal{I}$ with the adjoined zero. Obviously, for any
$\lambda\geqslant 2$ the Brandt $\lambda^0$-extension of the
semigroup ${\mathcal{I}}^0$ is isomorphic to the semigroup of
$I_\lambda\times I_\lambda$-matrix units and any Brandt
$\lambda^0$-extension of a semigroup with zero contains the
semigroup of $I_\lambda\times I_\lambda$-matrix units. Further by
$B_\lambda$ we shall denote the semigroup of $I_\lambda\times
I_\lambda$-matrix units and by $B^0_\lambda(1)$ the subsemigroup
of $I_\lambda\times I_\lambda$-matrix units of the Brandt
$\lambda^0$-extension of a monoid $S$ with zero. A completely
$0$-simple inverse semigroup is called a \emph{Brandt
semigroup}~\cite{Petrich1984}. A semigroup $S$ is a Brandt
semigroup if and only if $S$ is isomorphic to a Brandt
$\lambda$-extension $B_\lambda(G)$ of some group
$G$~\cite[Theorem~II.3.5]{Petrich1984}.

A non-trivial inverse semigroup is called a \emph{primitive
inverse semigroup} if all its non-zero idempotents are
primitive~\cite{Petrich1984}. A semigroup $S$ is a primitive
inverse semigroup if and only if $S$ is an orthogonal sum of
Brandt semigroups~\cite[Theorem~II.4.3]{Petrich1984}.

Green's relations $\mathscr{L}$, $\mathscr{R}$ and $\mathscr{H}$
on a semigroup $S$ are defined by:
\begin{itemize}
    \item[] $a\mathscr{L}b$ \; if and only if \; $a\cup Sa=b\cup Sb$;
    \item[] $a\mathscr{R}b$ \; if and only if \; $a\cup aS=b\cup bS$;
          \;  and
    \item[] $\mathscr{H}=\mathscr{L}\cap\mathscr{R}$,
\end{itemize}
for $a, b\in S$. For details about Green's relations see
\cite[\S~2.1]{CP} or \cite{Green1951}. We observe that two
non-zero elements $(\alpha_1,s,\beta_1)$ and
$(\alpha_2,t,\beta_2)$ of a Brandt semigroup $B_\lambda(G)$,
$s,t\in G$, $\alpha_1,\alpha_2,\beta_1,\beta_2\in I_\lambda$, are
$\mathscr{H}$-equivalent if and only if $\alpha_1=\alpha_2$ and
$\beta_1=\beta_2$ (see \cite[p.~93]{Petrich1984}).

By ${\mathscr S}$ we denote some class of topological semigroups.

\begin{definition}[\cite{GutikPavlyk2001, Stepp1969}]\label{def1}
A~semigroup $S\in{\mathscr S}$ is called {\it $H$-closed in
${\mathscr S}$}, if $S$ is a closed subsemigroup of any
topological semigroup $T\in{\mathscr S}$ which contains $S$ as a
subsemigroup. If ${\mathscr S}$ coincides with the class of all
topological semigroups, then the semigroup $S$ is called {\it
$H$-closed}.
\end{definition}

\begin{definition}[\cite{GutikPavlyk2003, Stepp1975}]\label{def3}
A~topological semigroup $S\in{\mathscr S}$ is called {\it
absolutely $H$-closed in the class ${\mathscr S}$} if any
continuous homomorphic image of $S$ into $T\in{\mathscr S}$ is
$H$-closed in ${\mathscr S}$. If ${\mathscr S}$ coincides with the
class of all topological semigroups, then the semigroup $S$ is
called {\it absolutely $H$-closed}.
\end{definition}

A~semigroup $S$ is called {\it algebraically closed in ${\mathscr
S}$} if $S$ with any semigroup topology $\tau$ is $H$-closed in
${\mathscr S}$ and $(S, \tau)\in{\mathscr
S}$~\cite{GutikPavlyk2001}. If ${\mathscr S}$ coincides with the
class of all topological semigroups, then the semigroup $S$ is
called {\it algebraically closed}. A~semigroup $S$ is called {\it
algebraically $h$-closed in ${\mathscr S}$} if $S$ with the
discrete topology $\mathfrak{d}$ is absolutely $H$-closed in
${\mathscr S}$ and $(S, \mathfrak{d})\in{\mathscr S}$. If
${\mathscr S}$ coincides with the class of all topological
semigroups, then the semigroup $S$ is called {\it algebraically
$h$-closed}.

Absolutely $H$-closed semigroups and algebraically $h$-closed
semigroups were introduced by Stepp in~\cite{Stepp1975}. There
they were called {\it absolutely maximal} and {\it algebraic
maximal}, respectively.

\begin{definition}[\cite{GutikPavlyk2001}]\label{definition3.6Brandt-ext}
Let $\lambda$ be a cardinal $\geqslant 1$ and
$(S,\tau)\in\mathscr{S}$. Let $\tau_{B}$ be a topology on
$B_{\lambda}(S)$ such that
\begin{itemize}
    \item[a)] $\left(B_{\lambda}(S), \tau_{B}\right)\in\mathscr{S}$;
              and
    \item[b)] $\tau_{B}|_{(\alpha, S^1, \alpha)}=\tau$ for some
    $\alpha\in I_{\lambda}$.
\end{itemize}
Then $\left(B_{\lambda}(S), \tau_{B}\right)$ is called a {\it
topological Brandt $\lambda$-extension of $(S, \tau)$ in
$\mathscr{S}$}. If $\mathscr{S}$ coincides with the class of all
topological semigroups, then $\left(B_{\lambda}(S),
\tau_{B}\right)$ is called a {\it topological Brandt
$\lambda$-extension of} $(S, \tau)$.
\end{definition}

\begin{definition}[\cite{GutikPavlyk2006}]\label{def2} Let
$\mathscr{S}_0$ be some class of topological semigroups with zero.
{Let $\lambda$ be a cardinal $\geqslant 1$ and
$(S,\tau)\in\mathscr{S}_0$. Let $\tau_{B}$ be a topology on
$B^0_{\lambda}(S)$ such that
\begin{itemize}
  \item[a)] $\left(B^0_{\lambda}(S),
            \tau_{B}\right)\in\mathscr{S}_0$;
  \item[b)] $\tau_{B}|_{(\alpha, S, \alpha)\cup\{ 0\}}=\tau$ for some
            $\alpha\in I_{\lambda}$.
\end{itemize}
Then $\left(B^0_{\lambda}(S), \tau_{B}\right)$ is called a {\it
topological Brandt $\lambda^0$-extension of $(S, \tau)$ in
$\mathscr{S}_0$.} If $\mathscr{S}_0$ coincides with the class of
all topological semigroups, then $\left(B^0_{\lambda}(S),
\tau_{B}\right)$ is called a {\it topological Brandt
$\lambda^0$-extension of} $(S, \tau)$.}
\end{definition}

Gutik and Pavlyk in \cite{GutikPavlyk2001} proved that the
following conditions for a topological semigroup $S$ are
equivalent:
\begin{itemize}
    \item[$(i)$] $S$ is an $H$-closed semigroup in the class of
                 topological inverse semigroups;
    \item[$(ii)$] there exists a cardinal $\lambda\geqslant 1$
                 such that any topological Brandt
                 $\lambda$-extension of $S$ is $H$-closed
                 in the class of topological inverse semigroups;
    \item[$(iii)$] for any cardinal $\lambda\geqslant 1$
                 every topological Brandt
                 $\lambda$-extension of $S$ is $H$-closed
                 in the class of topological inverse semigroups.
\end{itemize}
In \cite{GutikPavlyk2003} they showed that the similar statement
holds for absolutely $H$-closed topological semigroups in the
class of topological inverse semigroups.

In \cite{GutikPavlyk2006}, Gutik and Pavlyk proved the following:

\begin{theorem}\label{th4}
Let $S$ be a topological inverse monoid with zero. Then the
following conditions are equivalent:
\begin{itemize}
  \item[$(i)$] $S$ is an (absolutely) $H$-closed semigroup in the
  class of topological inverse semigroups;
  \item[$(ii)$] there exists a cardinal $\lambda\geqslant 1$ such that
  any topological Brandt $\lambda^0$-extension $B^0_{\lambda}(S)$ of
  the semigroup $S$ is (absolutely) $H$-closed in the class of
  topological inverse semigroups;
  \item[$(iii)$] for each cardinal $\lambda\geqslant 1$, every
  topological Brandt $\lambda^0$-extension $B^0_{\lambda}(S)$ of the
  semigroup $S$ is (absolutely) $H$-closed in the class of topological
  inverse semigroups.
\end{itemize}
\end{theorem}

Also, an example of an absolutely $H$-closed topological
semilattice $\mathscr{N}$ with zero and a topological Brandt
$\lambda^0$-extension $B^0_{\lambda}(\mathscr{N})$ of
$\mathscr{N}$ with the following properties:
\begin{itemize}
    \item[$(i)$] $B^0_{\lambda}(\mathscr{N})$ is an absolutely
                 $H$-closed semigroup for any infinite cardinal
                 $\lambda$;
    \item[$(ii)$] $B^0_{\lambda}(\mathscr{N})$ is a
                 $\sigma$-compact inverse topological semigroup
                 for any countable cardinal $\lambda$; and
    \item[$(iii)$] $B^0_{\lambda}(\mathscr{N})$ contains an
                 absolutely $H$-closed ideal $J$ such that the
                 Rees quotient semigroup
                 $B^0_{\lambda}(\mathscr{N})/J$ is not a
                 topological semigroup,
\end{itemize}
were constructed in \cite{GutikPavlyk2006}.

We observe that for any topological Brandt $\lambda$-extension
$B_\lambda(S)$ of a topological semigroup $S$ there exist a
topological monoid $T$ with zero and a topological Brandt
$\lambda^0$-extension $B^0_\lambda(T)$ of $T$, such that the
semigroups $B_\lambda(S)$ and $B^0_\lambda(T)$ are topologically
isomorphic. Algebraic properties of Brandt $\lambda^0$-extensions
of monoids with zero and non-trivial homomorphisms between Brandt
$\lambda^0$-extensions of monoids with zero, and a category whose
objects are ingredients of the construction of Brandt
$\lambda^0$-extensions of monoids with zeros were described in
\cite{GutikRepovs2010}. Also, in  \cite{GutikPavlykReiter2009} and
\cite{GutikRepovs2010} was described a category whose objects are
ingredients in the constructions of finite (compact, countably
compact) topological Brandt $\lambda^0$-extensions of topological
monoids with zeros.

In \cite{GutikPavlyk2001} and \cite{Pavlyk2004} for every infinite
cardinal $\lambda$, semigroup topologies on Brandt
$\lambda$-extensions which preserve an $H$-closedness and an
absolute $H$-closedness were constructed. An example of a
non-$H$-closed topological inverse semigroup $S$  in the class of
topological inverse semigroups such that for any cardinal
$\lambda\geqslant 1$  there exists an absolute $H$-closed
topological Brandt $\lambda$-extension of the semigroup $S$ in the
class of topological semigroups was constructed in
\cite{Pavlyk2004}.

In this paper we study (countably) compact and (absolutely)
$H$-closed primitive topological inverse semigroups. We describe
the structure of compact and countably compact primitive
topological inverse semigroups and show that any countably compact
primitive topological inverse semigroup embeds into a compact
primitive topological inverse semigroup.


\begin{lemma}\label{lemma3.1}
Let $E$ be a topological semilattice with zero $0$ such that every
non-zero idempotent of $E$ is primitive. Then every non-zero
element of $E$ is an isolated point in $E$.
\end{lemma}

\begin{proof}
Let $x\in E^*$. Since $E$ is a Hausdorff topological semilattice,
for every open neighbourhood $U(x)\not\ni 0$ of the point $x$
there exists an open neighbourhood $V(x)$ of $x$ such that
$V(x)\cdot V(x)\subseteq U(x)$. If $x$ is not an isolated point of
$E$ then $V(x)\cdot V(x)\ni 0$ which contradicts to the choice of
$U(x)$. This implies the assertion of the lemma.
\end{proof}

\begin{lemma}\label{lemma3.2}
Let $S$ be a primitive inverse topological semigroup and $S$ be an
orthogonal sum of the family
$\{B_{\lambda_{i}}(G_i)\}_{i\in\mathscr{A}}$ of topological Brandt
semigroups with zeros, i.~e.
$S=\sum_{i\in\mathscr{A}}B_{\lambda_{i}}(G_i)$. Let
$(\alpha_i,g_i,\beta_i)\in B_{\lambda_{i}}(G_i)$ be a non-zero
element of $S$. Then
\begin{itemize}
    \item[$(i)$] there exists an open neighbourhood $U$ of
    $(\alpha_i,g_i,\beta_i)$ such that $U\subseteq
    S^*_{\alpha_i,\beta_i}\subseteq B_{\lambda_{i}}(G_i)$;

    \item[$(ii)$] every non-zero idempotent of $S$ is an isolated
    point in $E(S)$.
\end{itemize}
\end{lemma}

\begin{proof}
$(i)$ Suppose to the contrary that $U\nsubseteq
S^*_{\alpha_i,\beta_i}\subseteq B_{\lambda_{i}}(G_i)$ for any open
neighbourhood $U$ of the point $(\alpha_i,g_i,\beta_i)$. Since $S$
is a Hausdorff space there exists an open neighbourhood $V$ of the
point $(\alpha_i,g_i,\beta_i)$ such that $0\notin V$. The
continuity of the semigroup operation in $S$ implies that there
exists an open neighbourhood $W$ of the point
$(\alpha_i,g_i,\beta_i)$ such that $(\alpha_i,1_i,\alpha_i)\cdot
W\cdot (\beta_i,1_i,\beta_i)\subseteq V$. Since $W\nsubseteq
S^*_{\alpha_i,\beta_i}$, we have that $0\in V$, a contradiction.

Statement $(ii)$ follows from Lemma~\ref{lemma3.1}.
\end{proof}

Lemma~\ref{lemma3.2} implies

\begin{corollary}\label{corollary3.3}
Every non-zero $\mathscr{H}$-class of a primitive inverse
topological semigroup $S$ is an open subset in $S$.
\end{corollary}

\begin{lemma}\label{lemma3.4}
If $S$ is a primitive topological inverse semigroup, then every
non-zero $\mathscr{H}$-class of $S$ is a clopen subset in $S$.
\end{lemma}

\begin{proof}
Let $H(e,f)$ be a non-zero $\mathscr{H}$-class in $S$ for
$e,f\in(E(S))^*$, i.~e.
\begin{equation*}
H(e,f)=\{ x\in S\mid x\cdot x^{-1}=e \; \mbox{~and~}\; x^{-1}\cdot
x=f\}.
\end{equation*}
Since $S$ is a topological inverse semigroup, the maps
$\varphi\colon S\rightarrow E(S)$ and $\psi\colon S\rightarrow
E(S)$ defined by the formulae $\varphi(x)=x\cdot x^{-1}$ and
$\psi(x)=x^{-1}\cdot x$ are continuous. By Lemma~\ref{lemma3.1},
$e$ and $f$ are isolated points in $E(S)$. Then the continuity of
the maps $\varphi$ and $\psi$ implies the statement of the lemma.
\end{proof}

The following example shows that the statement of
Lemma~\ref{lemma3.4} does not hold for primitive inverse locally
compact $H$-closed topological semigroups.

\begin{example}\label{example3.5}
Let $\mathbb{Z}$ be the discrete additive group of integers. We
extend the semigroup operation from $\mathbb{Z}$ onto
$\mathbb{Z}^0=\mathbb{Z}\cup\{\infty\}$ as follows:
\begin{equation*}
    x\cdot\infty=\infty\cdot x=\infty\cdot\infty=\infty, \qquad
    \mbox{for all} \quad x\in\mathbb{Z}.
\end{equation*}
We observe that $\mathbb{Z}^0$ is the group with adjoined zero
$\infty$. We determine a semigroup topology $\tau$ on
$\mathbb{Z}^0$ as follows:
\begin{itemize}
    \item[$(i)$] every non-zero element of $\mathbb{Z}^0$ is an
    isolated point;
    \item[$(ii)$] the family $\mathscr{B}(\infty)=\big\{
    U_n=\{\infty\}\cup\{x\in\mathbb{Z}\mid x\geqslant n\}\mid n
    \mbox{~is a positive integer}\big\}$ is a base of the topology
    $\tau$ at the point $\infty$.
\end{itemize}
A simple verification shows that $(\mathbb{Z}^0,\tau)$ is a
primitive inverse locally compact topological semigroup.
\end{example}

\begin{proposition}\label{proposition3.6}
$(\mathbb{Z}^0,\tau)$ is an $H$-closed topological semigroup.
\end{proposition}

\begin{proof}
Suppose that $\mathbb{Z}\sp{0}$ is embedded into a topological
semigroup $T$. If $\{n\sb{\iota}\}$ is a net in $\mathbb{N}$ for
which $\{-n\sb{\iota}\}$ converges in $T$ to $t\in
T\setminus\mathbb{Z}\sp{0}$, then the equation
$-n\sb{\iota}+(n\sb{\iota}+k)=k$ implies that $t\cdot\infty=k$ for
every $k\in\mathbb{N}$~--- which is impossible. So
$\mathbb{Z}\sp{0}$ is closed in $T$.
\end{proof}

\begin{proposition}\label{proposition3.7}
Every completely $0$-simple topological inverse semigroup $S$ is
topologically isomorphic to a topological Brandt
$\lambda$-extension $B_{\lambda}(G)$ of some topological group $G$
and cardinal $\lambda\geqslant 1$ in the class of topological
inverse semigroups. Furthermore:
\begin{itemize}
    \item[$(i)$] any non-zero subgroup of $S$ is topologically
         isomorphic to $G$ and every non-zero $\mathscr{H}$-class
         of $S$ is homeomorphic to $G$ and is a clopen subset in
         $S$;
    \item[$(ii)$] the family $\mathscr{B}(\alpha,g,\beta)=\{(\alpha,
         g\cdot U,\beta)\mid U\in\mathscr{B}_G(e)\}$, where
         $\mathscr{B}_G(e)$ is a base of the topology at the unity
         $e$ of $G$, is a base of the topology at the non-zero
         element $(\alpha,g,\beta)\in B_\lambda(G)$.
\end{itemize}
\end{proposition}

\begin{proof}
Let $G$ be a non-zero subgroup of $S$. Then by Theorem~3.9
of~\cite{CP} the semigroup $S$ is isomorphic to the Brandt
$\lambda$-extension of the subgroup $G$ for some cardinal
$\lambda\geqslant 1$. Since $S$ is a topological inverse semigroup
we have that $G$ is a topological group.

$(i)$ Let $e$ be the unity of $G$. We fix arbitrary $\alpha,
\beta, \gamma, \delta\in I_{\lambda}$ and define the maps
$\varphi^{\gamma\delta}_{\alpha\beta}\colon
B_{\lambda}(G)\rightarrow B_{\lambda}(G)$ and
$\varphi^{\alpha\beta}_{\gamma\delta}\colon
B_{\lambda}(G)\rightarrow B_{\lambda}(G)$ by the formulae
 $
 \varphi^{\gamma\delta}_{\alpha\beta}(s)=(\gamma, e, \alpha)\cdot
 s\cdot(\beta, e, \delta)
 $  and
 $
 \varphi_{\gamma\delta}^{\alpha\beta}(s)=(\alpha, e,
\gamma)\cdot s\cdot(\delta, e, \beta),
 $
$s\in B_{\lambda}(G)$. We observe that
 $
 \varphi_{\gamma\delta}^{\alpha\beta}\big(
\varphi^{\gamma\delta}_{\alpha\beta}\big((\alpha, x,
\beta)\big)\big)=(\alpha, x, \beta)
 $
  and\break
 $
 \varphi^{\gamma\delta}_{\alpha\beta}\big(
\varphi_{\gamma\delta}^{\alpha\beta}\big((\gamma, x, \delta)\big)
\big)=(\gamma, x, \delta)
 $
for all  $\alpha, \beta, \gamma, \delta\in I_{\lambda}$, $x\in G$,
and hence the restrictions
$\varphi^{\gamma\delta}_{\alpha\beta}\mid_{(\alpha,G,\beta)}$ and
$\varphi_{\gamma\delta}^{\alpha\beta}\mid_{(\gamma,G,\delta)}$ are
mutually invertible. Since the maps
$\varphi^{\gamma\delta}_{\alpha\beta}$ and
$\varphi_{\gamma\delta}^{\alpha\beta}$ are continuous on
$B_{\lambda}(G)$, the map
$\varphi^{\gamma\delta}_{\alpha\beta}\mid_{(\alpha,G,\beta)}\colon
(\alpha,G,\beta)\rightarrow (\gamma,G,\delta)$ is a homeomorphism
and the map
$\varphi^{\gamma\gamma}_{\alpha\alpha}\mid_{(\alpha,G,\alpha)}\colon$
$(\alpha,G,\alpha)\rightarrow (\gamma,G,\gamma)$ is a topological
isomorphism. We observe that the subset $(\alpha,G,\beta)$ of
$B_{\lambda}(G)$ is an $\mathscr{H}$-class of $B_{\lambda}(G)$ and
$(\alpha,G,\alpha)$ is  a subgroup of $B_{\lambda}(G)$ for all
$\alpha,\beta\in I_\lambda$. This completes the proof of assertion
$(i)$.

$(ii)$ The statement follows from assertion $(i)$ and Theorem~4.3
of \cite{HewittRoss1963}.
\end{proof}

We observe that Example~\ref{example3.5} implies that the
statements of Proposition~\ref{proposition3.7} are not true for
completely $0$-simple inverse topological semigroups.
Definition~\ref{definition3.6Brandt-ext} implies that $S$ is a
topological Brandt $\lambda$-extension $B_{\lambda}(G)$ of the
topological group $G$.

Gutik and Repov\v{s}, in \cite{GutikRepovs2007}, studied the
structure of $0$-simple countably compact topological inverse
semigroups. They proved that any $0$-simple countably compact
topological inverse semigroup is topologically isomorphic to a
topological Brandt $\lambda$-extension $B_{\lambda}(H)$ of a
countably compact topological group $H$ in the class of
topological inverse semigroups for some finite cardinal
$\lambda\geqslant 1$. This implies Pavlyk's Theorem (see
\cite{PavlykPhD}) on the structure of $0$-simple compact
topological inverse semigroups: \emph{every $0$-simple compact
topological inverse semigroup is topologically isomorphic to a
topological Brandt $\lambda$-extension $B_{\lambda}(H)$ of a
compact topological group $H$ in the class of topological inverse
semigroups for some finite cardinal $\lambda\geqslant 1$}.

The following theorem describes the structure of primitive
countably compact topological inverse semigroups.

\begin{theorem}\label{theorem3.8}
Every primitive countably compact topological inverse semigroup
$S$ is topologically isomorphic to an orthogonal sum
$\sum_{i\in\mathscr{A}}B_{\lambda_{i}}(G_i)$ of topological Brandt
$\lambda_i$-extensions $B_{\lambda_i}(G_i)$ of countably compact
topological groups $G_i$ in the class of topological inverse
semigroups for some finite cardinals $\lambda_i\geqslant 1$.
Moreover the family
\begin{equation*}
 \mathscr{B}(0)=\big\{S\setminus\big(
B_{\lambda_{i_1}}(G_{i_1})\cup
B_{\lambda_{i_2}}(G_{i_2})\cup\cdots\cup
B_{\lambda_{i_n}}(G_{i_n})\big)^*\mid i_1,
i_2,\ldots,i_n\in\mathscr{A}, n\in\mathbb{N}\big\}
\end{equation*}
determines a base of the topology at zero $0$ of $S$.
\end{theorem}

\begin{proof}
By Theorem~II.4.3 of~\cite{Petrich1984} the semigroup $S$ is an
orthogonal sum of Brandt semigroups and hence $S$ is an orthogonal
sum $\sum_{i\in\mathscr{A}}B_{\lambda_{i}}(G_i)$ of Brandt
$\lambda_i$-extensions $B_{\lambda_i}(G_i)$ of groups $G_i$. We
fix any $i_0\in\mathscr{A}$. Since $S$ is a topological inverse
semigroup, Proposition~II.2~\cite{EberhartSelden1969} implies that
$B_{\lambda_{i_0}}(G_{i_0})$ is a topological inverse semigroup.
By Proposition~\ref{proposition3.7}, $B_{\lambda_{i_0}}(G_{i_0})$
is a closed subsemigroup of $S$ and hence by
Theorem~3.10.4~\cite{Engelking1989}, $B_{\lambda_{i_0}}(G_{i_0})$
is a countably compact $0$-simple topological inverse semigroup.
Then, by Theorem~2 of \cite{GutikRepovs2007}, the semigroup
$B_{\lambda_{i_0}}(G_{i_0})$ is a topological Brandt
$\lambda_i$-extension of countably compact topological group
$G_{i_0}$ in the class of topological inverse semigroups for some
finite cardinal $\lambda_{i_0}\geqslant 1$. This completes the
proof of the first assertion of the theorem.

Suppose on the contrary that $\mathscr{B}(0)$ is not a base at
zero $0$ of $S$. Then, there exists an open neighbourhood $U(0)$
of zero $0$ such  that
$U(0)\bigcup\big(B_{\lambda_{i_1}}(G_{i_1})\cup
B_{\lambda_{i_2}}(G_{i_2})\cup\cdots\cup
B_{\lambda_{i_n}}(G_{i_n})\big)^*\neq S$ for finitely many indexes
$i_1, i_2,\ldots,i_n\in\mathscr{A}$. Therefore there exists an
infinitely family $\mathscr{F}$ of non-zero disjoint
$\mathscr{H}$-classes such that $H\nsubseteq U(0)$ for all
$H\in\mathscr{F}$. Let $\mathscr{F}_0$ be an infinite countable
subfamily of $\mathscr{F}$. We put $W=\bigcup\{H\mid
H\in\mathscr{F}\setminus\mathscr{F}_0\}$. Lemma~\ref{lemma3.4}
implies that the family $\mathscr{C}=\{U(0), W\}\cup
\mathscr{F}_0$ is an open countable cover of $S$. Simple
observation shows that the cover $\mathscr{C}$ does not contains a
finite subcover. This contradicts to the countable compactness of
$S$. The obtained contradiction implies the last assertion of the
theorem.
\end{proof}

Since any maximal subgroup of a compact topological semigroup $T$
is a compact subset in $T$ (see \cite[Vol.~1, Theorem~1.11]{CHK})
Theorem~\ref{theorem3.8} implies the following:

\begin{corollary}\label{corollary3.9}
Every primitive compact topological inverse semigroup $S$ is
topologically isomorphic to an orthogonal sum
$\sum_{i\in\mathscr{A}}B_{\lambda_{i}}(G_i)$ of topological Brandt
$\lambda_i$-extensions $B_{\lambda_i}(G_i)$ of compact topological
groups $G_i$ in the class of topological inverse semigroups for
some finite cardinals $\lambda_i\geqslant 1$ and the family
\begin{equation*}
 \mathscr{B}(0)=\big\{S\setminus\big(
B_{\lambda_{i_1}}(G_{i_1})\cup
B_{\lambda_{i_2}}(G_{i_2})\cup\cdots\cup
B_{\lambda_{i_n}}(G_{i_n})\big)^*\mid i_1,
i_2,\ldots,i_n\in\mathscr{A}, n\in\mathbb{N}\big\}
\end{equation*}
determines a base of the topology at zero $0$ of $S$.
\end{corollary}

\begin{theorem}\label{theorem3.10}
Every primitive countably compact topological inverse semigroup
$S$ is a dense subsemigroup of a primitive compact topological
inverse semigroup.
\end{theorem}

\begin{proof}
By Theorem~\ref{theorem3.8} the topological semigroup $S$ is
topologically isomorphic to an orthogonal sum
$\sum_{i\in\mathscr{A}}B_{\lambda_{i}}(G_i)$ of topological Brandt
$\lambda_i$-extensions $B_{\lambda_i}(G_i)$ of countably compact
topological groups $G_i$ in the class of topological inverse
semigroups for some finite cardinals $\lambda_i\geqslant 1$. Since
any countably compact topological group $G_i$ is pseudocompact,
the Comfort-Ross Theorem (see \cite[Theorem~4.1]{ComfortRoss1966})
implies that the Stone-\v{C}ech compactification $\beta(G_i)$ is a
compact topological group and the inclusion mapping $f_i$ of $G_i$
into $\beta(G_i)$ is a topological isomorphism for all
$i\in\mathscr{A}$. On the orthogonal sum
$\sum_{i\in\mathscr{A}}B_{\lambda_{i}}(G_i)$ of Brandt
$\lambda$-extensions $B_{\lambda_{i}}(\beta(G_i))$,
$i\in\mathscr{A}$, we determine a topology $\tau$ as follows:
\begin{itemize}
    \item[$(a)$] the family
         $\mathscr{B}(\alpha_i,g_i,\beta_i)=\{(\alpha_i,
         g_i\cdot U,\beta_i)\mid
         U\in\mathscr{B}_{\beta(G_i)}(e_i)\}$ is a base of the topology at the non-zero
         element $(\alpha_i,g_i,\beta_i)\in
         B_{\lambda_i}(\beta(G_i))$, where
         $\mathscr{B}_{\beta(G_i)}(e_i)$ is a base of the topology
         at the unity $e_i$ of the compact topological group
         $\beta(G_i)$; and
    \item[$(b)$] the family
\begin{equation*}
 \mathscr{B}(0)=\big\{S\setminus\big(
B_{\lambda_{i_1}}(\beta(G_{i_1}))\cup
B_{\lambda_{i_2}}(\beta(G_{i_2}))\cup\cdots\cup
B_{\lambda_{i_n}}(\beta(G_{i_n}))\big)^*\mid i_1,
i_2,\ldots,i_n\in\mathscr{A}, n\in\mathbb{N}\big\}
\end{equation*}
determines a base of the topology at zero $0$ of
$\sum_{i\in\mathscr{A}}B_{\lambda_{i}}(G_i)$.
\end{itemize}
By Theorem~II.4.3 of~\cite{Petrich1984},
$\sum_{i\in\mathscr{A}}B_{\lambda_{i}}(\beta(G_i))$ is a primitive
inverse semigroup and simple verifications show that
$\sum_{i\in\mathscr{A}}B_{\lambda_{i}}(\beta(G_i))$ with the
topology $\tau$ is a compact topological inverse semigroup.

We define a map $f\colon\sum_{i\in\mathscr{A}}B_{\lambda_{i}}(G_i)
\rightarrow\sum_{i\in\mathscr{A}}B_{\lambda_{i}}(\beta(G_i))$ as
follows:
\begin{equation*}
    f(0)=0 \qquad \mbox{and} \qquad
    f((\alpha_i,g_i,\beta_i))=(\alpha_i,f_i(g_i),\beta_i)\in
    B_{\lambda_{i}}(\beta(G_i)) \quad \mbox{for} \quad
    (\alpha_i,g_i,\beta_i)\in B_{\lambda_{i}}(G_i).
\end{equation*}
Simple verifications show that $f$ is a continuous homomorphism.
Since $f_i\colon G_i\rightarrow\beta(G_i)$ is a topological
isomorphism, we have that
$f\colon\sum_{i\in\mathscr{A}}B_{\lambda_{i}}(G_i)
\rightarrow\sum_{i\in\mathscr{A}}B_{\lambda_{i}}(\beta(G_i))$ is a
topological isomorphism too.
\end{proof}

Gutik and Repov\v{s} in \cite{GutikRepovs2007} showed that the
Stone-\v{C}ech compactification $\beta(T)$ of a $0$-simple
countably compact topological inverse semigroup $T$ is a
$0$-simple compact topological inverse semigroup. In this context
the following question arises naturally:

\begin{question}
Is the Stone-\v{C}ech compactification $\beta(T)$ of a primitive
countably compact topological inverse semigroup $T$ a topological
semigroup (a primitive topological inverse semigroup)?
\end{question}


\begin{theorem}\label{theorem3.21}
Let $S=\bigcup_{\alpha\in\mathscr{A}}S_{\alpha}$ be a topological
inverse semigroup such that
\begin{itemize}
    \item[$(i)$] $S_{\alpha}$ is an $H$-closed (resp., absolutely
         $H$-closed) semigroup in the class of topological inverse
         semigroups for any $\alpha\in\mathscr{A}$; and
    \item[$(ii)$] there exists an $H$-closed (resp., absolutely
         $H$-closed) subsemigroup $T$ of $S$ in the class of
         topological inverse semigroups such that $S_{\alpha}\cdot
         S_{\beta}\subseteq T$ for all $\alpha\neq\beta$,
         $\alpha,\beta\in\mathscr{A}$.
\end{itemize}
Then $S$ is an $H$-closed (resp., absolutely $H$-closed) semigroup
in the class of topological inverse semigroups.
\end{theorem}

\begin{proof}
We consider the case of absolute $H$-closedness only.

Suppose on the contrary that there exist a topological inverse
semigroup $G$ and a continuous homomorphism  $h\colon S\rightarrow
G$ such that $h(S)$ is not closed subsemigroup in $G$. Without
loss of generality we can assume that
$\operatorname{cl}_G(h(S))=G$. Thus, by Proposition~II.2 of
\cite{EberhartSelden1969}, $G$ is a topological inverse semigroup.

Then, $G\setminus h(S)\neq\varnothing$. Let $x\in G\setminus
h(S)$. Since $S$ and $G$ are topological inverse semigroups we
have that $h(S)$ is an inverse subsemigroup in $G$ and hence
$x^{-1}\in G\setminus h(S)$. The semigroup $T$ is an absolutely
$H$-closed semigroup in the class of topological inverse
semigroups implies that there exists an open neighbourhood $U(x)$
of the point $x$ in $T$ such that $U(x)\cap h(T)=\varnothing$.
Since $G$ is a topological inverse semigroup there exist open
neighbourhoods $V(x)$ and $V(x^{-1})$ of the points $x$ and
$x^{-1}$ in $G$, respectively, such that $V(x)\cdot V(x^{-1})\cdot
V(x)\subseteq U(x)$. But
$x,x^{-1}\in\operatorname{cl}_G(h(S))\setminus h(S)$ and since
$\{S_{\alpha}\mid \alpha\in\mathscr{A}\}$ is the family of
absolutely $H$-closed semigroups in the class of topological
inverse semigroups, each of the neighbourhoods $V(x)$ and
$V(x^{-1})$ intersects infinitely many subsemigroups
$h(S_{\beta})$ in $G$, $\beta\in\mathscr{A}$. Hence,
$\big(V(x)\cdot V(x^{-1})\cdot V(x)\big)\cap h(T)\neq\varnothing$.
This contradicts the assumption that $U(x)\cap h(T)=\varnothing$.
The obtained contradiction implies that $S$ is an absolutely
$H$-closed semigroup in the class of topological inverse
semigroups.

The proof in the case of $H$-closeness is similar to the previous
one.
\end{proof}

Theorem~\ref{theorem3.21} implies:

\begin{corollary}\label{corollary3.22}
Let $S=\bigcup_{\alpha\in\mathscr{A}}S_{\alpha}$ be an inverse
semigroup such that
\begin{itemize}
    \item[$(i)$] $S_{\alpha}$ is an algebraically closed
         (resp., algebraically $h$-closed)
         semigroup in the class of topological inverse semigroups
         for any $\alpha\in\mathscr{A}$; and
    \item[$(ii)$] there exists an algebraically closed
         (resp., algebraically $h$-closed)
         subsemigroup $T$ of $S$ in the class of topological
         inverse semigroups such that $S_{\alpha}\cdot
         S_{\beta}\subseteq T$ for all $\alpha\neq\beta$,
         $\alpha,\beta\in\mathscr{A}$.
\end{itemize}
Then $S$ is an algebraically closed (resp., algebraically
$h$-closed) semigroup in the class of topological inverse
semigroups.
\end{corollary}

Theorem~\ref{theorem3.21} implies:

\begin{theorem}\label{theorem3.23}
Let a topological inverse semigroup $S$ be an orthogonal sum of
the family $\{ S_{\alpha}\}_{\alpha\in\mathscr{A}}$ of $H$-closed
(resp., absolutely $H$-closed) topological inverse semigroups with
zeros in the class of topological inverse semigroups. Then $S$ is
an $H$-closed (resp., absolutely $H$-closed) topological inverse
semigroup in the class of topological inverse semigroups.
\end{theorem}

Corollary~\ref{corollary3.22} implies

\begin{corollary}\label{corollary3.24}
Let an inverse semigroup $S$ be an orthogonal sum of the family
$\{ S_{\alpha}\}_{\alpha\in\mathscr{A}}$ of algebraically closed
(resp., algebraically $h$-closed) inverse semigroups with zeros in
the class of topological inverse semigroups. Then $S$ is an
algebraically closed (resp., algebraically $h$-closed) inverse
semigroup in the class of topological inverse semigroups.
\end{corollary}

Recall \cite{Aleksandrov1942}, a topological group $G$ is called
\emph{absolutely closed} if $G$ is a closed subgroup of any
topological group which contains $G$ as a subgroup. In our
terminology such topological groups are called $H$-closed in the
class of topological groups. In \cite{Raikov1946} Raikov proved
that a topological group $G$ is absolutely closed if and only if
it is Raikov complete, i.~e. $G$ is complete with respect to the
two sided uniformity.

A topological group $G$ is called \emph{$h$-complete} if for every
continuous homomorphism $f\colon G\rightarrow H$ into a
topological group $H$ the subgroup $f(G)$ of $H$ is
closed~\cite{DikranjanUspenskij1998}. The $h$-completeness is
preserved under taking products and closed central
subgroups~\cite{DikranjanUspenskij1998}.

Gutik and Pavlyk in \cite{GutikPavlyk2003} showed that a
topological group $G$ is $H$-closed (resp., absolutely $H$-closed)
in the class of topological inverse semigroups if and only if $G$
is absolutely closed (resp., $h$-complete).

\begin{theorem}\label{theorem3.25}
For a primitive topological inverse semigroup $S$ the following
assertions are equivalent:
\begin{itemize}
    \item[$(i)$] every maximal subgroup of $S$ is absolutely
         closed;
    \item[$(ii)$] the semigroup $S$ with every inverse semigroup
     topology $\tau$ is $H$-closed in the class of topological
     inverse semigroups.
\end{itemize}
\end{theorem}

\begin{proof}
$(i)\Rightarrow(ii)$ Suppose that a primitive topological inverse
semigroup $S$ is an orthogonal sum
$\sum_{i\in\mathscr{A}}B_{\lambda_{i}}(G_i)$ of topological Brandt
$\lambda_i$-extensions $B_{\lambda_i}(G_i)$ of topological groups
$G_i$ in the class of topological inverse semigroups and every
topological group $G_i$ is absolutely closed. Then, by Theorem~3
of \cite{GutikPavlyk2001} any topological Brandt
$\lambda_i$-extension $B_{\lambda_i}(G_i)$ of topological group
$G_i$ is $H$-closed in the class of topological inverse
semigroups. Theorem~\ref{theorem3.23} implies that $S$ is an
$H$-closed topological inverse semigroup in the class of
topological inverse semigroups.

$(ii)\Rightarrow(i)$  Let $G$ be any maximal non-zero subgroup of
$S$. Since $S$ is a primitive topological inverse semigroup we
have that $S$ is an orthogonal sum
$\sum_{i\in\mathscr{A}}B_{\lambda_{i}}(G_i)$ of Brandt
$\lambda$-extensions $B_{\lambda_{i}}(G_i)$ of topological groups
$G_i$ and hence there exists a topological Brandt
$\lambda_{i_0}$-extension $B_{\lambda_{i_0}}(G_{i_0})$,
$i\in\mathscr{A}$, such that $B_{\lambda_{i_0}}(G_{i_0})$ contains
the maximal subgroup $G$ and $B_{\lambda_{i_0}}(G_{i_0})$ is a
subsemigroup of $S$.

Suppose on the contrary that the topological group
$G=G\sb{i\sb{0}}$ is not absolutely closed. Then there exists a
topological group $H$ which contains $G$ as a dense proper
subgroup. For every $i\in\mathscr{A}$ we put
\begin{equation*}
    H\sb{i}=
\left\{%
\begin{array}{ll}
    G\sb{i}, & \hbox{if~}\; i\neq i\sb{0}; \\
    H,       & \hbox{if~}\; i=i\sb{0}. \\
\end{array}%
\right.
\end{equation*}
On the orthogonal sum $\sum_{i\in\mathscr{A}}B_{\lambda_{i}}(H_i)$
of Brandt $\lambda$-extensions $B_{\lambda_{i}}(H_i)$,
$i\in\mathscr{A}$, we determine a topology $\tau\sb{0}$ as
follows:
\begin{itemize}
    \item[$(a)$] the family
         $\mathscr{B}(\alpha_i,g_i,\beta_i)=\{(\alpha_i,
         g_i\cdot U,\beta_i)\mid
         U\in\mathscr{B}_{H_i}(e_i)\}$ is a base of the topology at the non-zero
         element $(\alpha_i,g_i,\beta_i)\in
         B_{\lambda_i}(H_i)$, where $\mathscr{B}_{H_i}(e_i)$ is a base of the topology
         at the unity $e_i$ of the topological group $H_i$; and
    \item[$(b)$] the zero $0$ is an isolated point in
$\left(\sum_{i\in\mathscr{A}}B_{\lambda_{i}}(H_i),\tau\sb{0}\right)$.
\end{itemize}
By Theorem~II.4.3 of~\cite{Petrich1984},
$\sum_{i\in\mathscr{A}}B_{\lambda_{i}}(H_i)$ is a primitive
inverse semigroup and simple verifications show that
$\sum_{i\in\mathscr{A}}B_{\lambda_{i}}(H_i)$ with the topology
$\tau\sb{0}$ is a topological inverse semigroup. Also we observe
that the semigroup $\sum_{i\in\mathscr{A}}B_{\lambda_{i}}(G_i)$
with the induced from
$\left(\sum_{i\in\mathscr{A}}B_{\lambda_{i}}(H_i),\tau\sb{0}\right)$
topology is a topological inverse semigroup which is a dense
proper inverse subsemigroup of
$\left(\sum_{i\in\mathscr{A}}B_{\lambda_{i}}(H_i),\tau\sb{0}\right)$.
The obtained contradiction completes the statement of the theorem.
\end{proof}

Theorem~\ref{theorem3.25} implies

\begin{corollary}\label{corollary3.26}
For a primitive inverse semigroup $S$ the following assertions are
equivalent:
\begin{itemize}
    \item[$(i)$] every maximal subgroup of $S$ is algebraically
      closed in the class of topological inverse semigroups;
    \item[$(ii)$] the semigroup $S$ is algebraically
      closed in the class of topological inverse semigroups.
\end{itemize}
\end{corollary}

\begin{theorem}\label{theorem3.27}
For a primitive topological inverse semigroup $S$ the following
assertions are equivalent:
\begin{itemize}
    \item[$(i)$] every maximal subgroup of $S$ is $h$-complete;
    \item[$(ii)$] the semigroup $S$ with every inverse semigroup
     topology $\tau$ is absolutely $H$-closed in the class of
     topological inverse semigroups.
\end{itemize}
\end{theorem}

\begin{proof}
$(i)\Rightarrow(ii)$ Suppose that a primitive topological inverse
semigroup $S$ is an orthogonal sum
$\sum_{i\in\mathscr{A}}B_{\lambda_{i}}(G_i)$ of topological Brandt
$\lambda_i$-extensions $B_{\lambda_i}(G_i)$ of topological groups
$G_i$ in the class of topological inverse semigroups and every
topological group $G_i$ is $h$-complete. Then by Theorem~14 of
\cite{GutikPavlyk2003} any topological Brandt
$\lambda_i$-extension $B_{\lambda_i}(G_i)$ of topological group
$G_i$ is absolutely $H$-closed in the class of topological inverse
semigroups. Theorem~\ref{theorem3.23} implies that $S$ is an
absolutely $H$-closed topological inverse semigroup in the class
of topological inverse semigroups.

$(ii)\Rightarrow(i)$ Let $G$ be any maximal non-zero subgroup of
$S$. Since $S$ is a primitive topological inverse semigroup, $S$
is an orthogonal sum $\sum_{i\in\mathscr{A}}B_{\lambda_{i}}(G_i)$
of Brandt $\lambda$-extensions $B_{\lambda_{i}}(G_i)$ of
topological groups $G_i$. Hence there exists a topological Brandt
$\lambda_{i_0}$-extension $B_{\lambda_{i_0}}(G_{i_0})$,
$i\in\mathscr{A}$, such that $B_{\lambda_{i_0}}(G_{i_0})$ contains
the maximal subgroup $G$ and $B_{\lambda_{i_0}}(G_{i_0})$ is a
subsemigroup of $S$.

Suppose on the contrary that the topological group
$G=G\sb{i\sb{0}}$ is not $h$-completed. Then there exist a
topological group $H$ and continuous homomorphism $h\colon
G\rightarrow H$ such that $h(G)$ is a dense proper subgroup of
$H$. On the Brandt $\lambda$-extension $B_{\lambda_{i\sb{0}}}(H)$,
we determine a topology $\tau\sb{H}$ as follows:
\begin{itemize}
    \item[$(a)$] the family
         $\mathscr{B}(\alpha_{i\sb{0}},g_{i\sb{0}},\beta_{i\sb{0}})=\{(\alpha_{i\sb{0}},
         g_i\cdot U,\beta_{i\sb{0}})\mid
         U\in\mathscr{B}_{H}(e)\}$ is a base of the topology at the non-zero
         element $(\alpha_{i\sb{0}},g_i,\beta_{i\sb{0}})\in
         B_{\lambda_i}(H)$, where $\mathscr{B}_{H}(e)$ is a base of the topology
         at the unity $e$ of the topological group $H$; and
    \item[$(b)$] the zero $0$ is an isolated point in
$\left(B_{\lambda_{{i\sb{0}}}}(H),\tau\sb{H}\right)$.
\end{itemize}
Then $B_{\lambda_{{i\sb{0}}}}(H)$ is an inverse semigroup and
simple verifications show that $B_{\lambda_{{i\sb{0}}}}(H)$ with
the topology $\tau\sb{H}$ is a topological inverse semigroup.

On the orthogonal sum $\sum_{i\in\mathscr{A}}B_{\lambda_{i}}(G_i)$
of Brandt $\lambda$-extensions $B_{\lambda_{i}}(G_i)$,
$i\in\mathscr{A}$, we determine a topology $\tau\sb{\star}$ as
follows:
\begin{itemize}
    \item[$(a)$] the family
         $\mathscr{B}(\alpha_i,g_i,\beta_i)=\{(\alpha_i,
         g_i\cdot U,\beta_i)\mid
         U\in\mathscr{B}_{G_i}(e_i)\}$ is a base of the topology at the non-zero
         element $(\alpha_i,g_i,\beta_i)\in
         B_{\lambda_i}(G_i)$, where $\mathscr{B}_{G_i}(e_i)$ is a base of the topology
         at the unity $e_i$ of the topological group $G_i$; and
    \item[$(b)$] the zero $0$ is an isolated point in
$\left(\sum_{i\in\mathscr{A}}B_{\lambda_{i}}(G_i),\tau\sb{\star}\right)$.
\end{itemize}
By Theorem~II.4.3 of~\cite{Petrich1984},
$\sum_{i\in\mathscr{A}}B_{\lambda_{i}}(G_i)$ is a primitive
inverse semigroup and simple verifications show that
$\sum_{i\in\mathscr{A}}B_{\lambda_{i}}(G_i)$ with the topology
$\tau\sb{\star}$ is a topological inverse semigroup.

We define the map $f\colon S\rightarrow B_{\lambda_{i_0}}(H)$ as
follows:
\begin{equation*}
    f(x)=
    \left\{%
\begin{array}{ll}
    h(x), & \hbox{if}~x\in B_{\lambda_{i_0}}(G_{i_0});\\
    0, & \hbox{if}~x\notin B_{\lambda_{i_0}}(G_{i_0}), \\
\end{array}%
\right.
\end{equation*}
where $0$ is zero of $S$. Evidently the defined map $f$ is a
continuous homomorphism. Then $f(S)=B_{\lambda_{i_0}}(h(G_{i_0}))$
is a dense proper inverse subsemigroup of the topological inverse
semigroup $\left(B_{\lambda_{{i\sb{0}}}}(H),\tau\sb{H}\right)$.
The obtained contradiction completes the statement of the theorem.
\end{proof}

Theorem~\ref{theorem3.27} implies

\begin{corollary}\label{corollary3.28}
For a primitive inverse semigroup $S$ the following assertions are
equivalent:
\begin{itemize}
    \item[$(i)$] every maximal subgroup of $S$ is algebraically
      $h$-closed in the class of topological inverse semigroups;
    \item[$(ii)$] the semigroup $S$ is algebraically
      $h$-closed in the class of topological inverse semigroups.
\end{itemize}
\end{corollary}


\section*{Acknowledgements}

The authors are grateful to the referee for several comments and
suggestions which have considerably improved the original version
of the manuscript.




\begin{thebibliography}{29}

\bibitem{Aleksandrov1942} A. D. Aleksandrov, \emph{ On an extension of
Hausdorff space to $H$-closed}, Dokl. Akad. Nauk. SSSR \textbf{37}
(1942), 138---141 (in Russian).

\bibitem{CHK} J.~H.~Carruth, J.~A.~Hildebrant and  R.~J.~Koch,
\emph{The Theory of Topological Semigroups}, Vol. I, Marcel
Dekker, Inc., New York and Basel, 1983; Vol. II, Marcel Dekker,
Inc., New York and Basel, 1986.



\bibitem{CP} A.~H.~Clifford and  G.~B.~Preston, \emph{The
Algebraic Theory of Semigroups}, Vol. I., Amer. Math. Soc. Surveys
7, Providence, R.I., 1961; Vol. II., Amer. Math. Soc. Surveys 7,
Providence, R.I., 1967.

\bibitem{ComfortRoss1966} W.~W.~Comfort and K.~A.~Ross, \emph{
Pseudocompactness and uniform continuity in topological groups},
Pacif. J. Math. {\bf 16}:3 (1966), 483---496.

\bibitem{DikranjanUspenskij1998} D. Dikranjan and V. V. Uspenskij,
\emph{ Categorically compact topological groups}, J. Pure Appl.
Algebra \textbf{126} (1998), 149---168.

\bibitem{EberhartSelden1969} C. Eberhart and J. Selden, \emph{ On the
closure of the bicyclic semigroup}, Trans. Amer. Math. Soc. {\bf
144} (1969), 115---126.

\bibitem{Engelking1989} R.~Engelking, \emph{General Topology},
2nd ed., Heldermann, Berlin, 1989.

\bibitem{Green1951} J. A. Green, \emph{ On the structure of
semigroups}, Ann. of Math. \textbf{54} (1951), 163---172.

\bibitem{GutikPavlyk2001} O. V. Gutik and K.~P. Pavlyk,
\emph{H-closed topological semigroups and Brandt
$\lambda$-extensions}, Mat. Metody Phis.-Mech. Polya. {\bf 44}:3
(2001), 20---28 (in Ukrainian).

\bibitem{GutikPavlyk2003} O. V. Gutik and K.~P. Pavlyk,
\emph{Topological Brandt $\lambda$-extensions of absolutely
$H$-closed topological inverse semigroups}, Visnyk Lviv Univ. Ser.
Mech.-Math. \textbf{61} (2003), 98---105.

\bibitem{GutikPavlyk2005} O.~V.~Gutik and K.~P.~Pavlyk, \emph{On
topological semigroups of matrix units}, Semigroup Forum {\bf 71}
(2005), 389---400.


\bibitem{GutikPavlyk2006} {O. V. Gutik and K.~P. Pavlyk,} \emph{On
Brandt  $\lambda^0$-extensions of semigroups with zero}, Mat.
Metody Phis.-Mech. Polya. {\bf 49}:3 (2006), 26---40.


\bibitem{GutikPavlykReiter2009} O. Gutik, K. Pavlyk and A. Reiter,
\emph{Topological semigroups of matrix units and countably compact
Brandt $\lambda^0$-extensions}, Mat. Stud. 32:2 (2009), 115---131.

\bibitem{GutikRepovs2007} O.~Gutik and D.~Repov\v{s}, \emph{On
countably compact $0$-simple topological inverse semigroups},
Semigroup Forum \textbf{75}:2 (2007), 464---469.

\bibitem{GutikRepovs2010} O.~Gutik and
D.~Repov\v{s}, \emph{On Brandt $\lambda^0$-extensions of monoids
with zero}, Semigroup Forum \textbf{80}:1 (2010), 8---32.


\bibitem{HewittRoss1963} E.~Hewitt and K. A. Ross, \emph{Abstract
Harmonic Analysis}, Vol.~1, Springer, Berlin, 1963.

\bibitem{Howie1995} J. M. Howie,
\emph{Fundamentals of Semigroup Theory}, London Math. Monographs,
New Ser. 12,  Clarendon Press, Oxford, 1995.

\bibitem{Pavlyk2004} K. P. Pavlyk, \emph{Absolutely $H$-closed
topological semigroups and Brandt $\lambda$-extensions}, Applied
Problems of Mechanics and Mathematics \textbf{2} (2004), 61---68
(in Ukrainian).

\bibitem{PavlykPhD} K. P. Pavlyk, \emph{Topological semigroups of
matrix units and Brandt $\lambda$-extensions of topological
semigroups}, PhD Thesis, Lviv University, 2006 (in Ukrainian).

\bibitem{Petrich1984} M.~Petrich, \emph{Inverse Semigroups}, John
Wiley $\&$ Sons, New York, 1984.


\bibitem{Raikov1946} D. A. Raikov, {\em On a completion of
topological group}, Izv. Akad. Nauk. SSSR, Ser. Mat. \textbf{10}:6
(1946), 513---528 (in Rus\-sian).

\bibitem{Stepp1969} J.~W.~Stepp, {\it A note on maximal locally
compact semigroups,} Proc. Amer. Math. Soc. {\bf 20}:1 (1969),
251---253.

\bibitem{Stepp1975} J.~W.~Stepp, {\it Algebraic maximal
semilattices,} Pacific J. Math. {\bf 58}:1 (1975), 243---248.

\end{thebibliography}
\end{document}